\documentclass[runningheads, orivec]{llncs}

\usepackage[T1]{fontenc}
\usepackage{amsmath,amssymb}
\usepackage{graphicx}
\usepackage{tabularx}
\usepackage[hidelinks]{hyperref}
\usepackage[a4paper,margin=1in]{geometry}

\AddToHook{env/proof/end}{\qed}

\begin{document}

\title{Sharp Bounds on the Number of Small Cuts}
\author{Yu Cong \and Chao Xu}
\institute{University of Electronic Science and Technology of China\\
\email{yucong143@gmail.com, thechaoxu@gmail.com}}

\maketitle

\begin{abstract}
A $k$-cut of an undirected graph is a minimal edge set whose removal leaves at least $k$ connected components, and its cost is the sum of its edge costs.
Write $\lambda_k$ for the minimum cost of a $k$-cut.
For fixed $k\ge2$ and $\alpha\ge1$, we show that in any $n$-vertex graph with $n\ge k$ and nonnegative edge costs, the number of $k$-cuts of cost strictly below $\alpha\lambda_k$ is $O(n^{\lceil\alpha k\rceil-1})$.
This implies an $O(n^{\lfloor\alpha k\rfloor})$ bound for $k$-cuts of cost at most $\alpha\lambda_k$, sharpening the $O(n^{\alpha k})$ bound of Gupta, Harris, Lee and Li \cite{guptaHarrisLeeLi2022}.
For $k=2$, our strict bound resolves a conjecture of Henzinger and Williamson \cite{henzingerWilliamson1996} and an open question of Karger \cite{karger2000}.
The exponent $\lceil\alpha k\rceil-1$ is optimal for $k$-cuts for every $k\ge2$ and $\alpha>1$.
\end{abstract}

\section{Introduction}\label{sec:introduction}

In an $n$-vertex undirected graph with positive minimum-cut cost $\lambda_2$, there can be exponentially many cuts, yet at most $\binom n2$ minimum cuts \cite{lomonosov1971upper,dinitz1976structure}.
Bounds on the number of near-minimum cuts are central to cut enumeration, minimum-cut algorithms, and network reliability.
For fixed $\alpha\ge1$, Karger \cite{karger1993} used random contraction to bound the number of cuts of cost at most $\alpha\lambda_2$ by $O(n^{2\alpha})$.
For cuts of cost strictly below $\alpha\lambda_2$, Nagamochi, Nishimura and Ibaraki \cite{nagamochiNishimuraIbaraki1997} proved the bound $\binom n2$ at $\alpha=4/3$, and Henzinger and Williamson \cite{henzingerWilliamson1996} extended the quadratic bound $O(n^2)$ to $\alpha=3/2$.
They conjectured that whenever $\alpha>1$ is a half-integer, there are $O(n^{2\alpha-1})$ cuts of cost strictly below $\alpha\lambda_2$, equivalently $O(n^{\lceil2\alpha\rceil-1})$ for every fixed $\alpha>1$.
Their conjecture implies a \emph{stair-step} behavior rather than the continuous exponent $2\alpha$.
Karger \cite{karger2000} subsequently used tree packing to prove the stair-step bound $O(n^{\lfloor2\alpha\rfloor})$ for cuts of cost at most $\alpha\lambda_2$.
When $\alpha$ is not a half-integer, $\lfloor2\alpha\rfloor=\lceil2\alpha\rceil-1$, settling the strict bound away from half-integers.
At half-integers $\alpha>1$, however, tree packing gives only $O(n^{2\alpha})$, and Karger \cite[Section~3.2]{karger2000} asked whether this bound can be improved for cuts of cost strictly below $\alpha\lambda_2$ beyond $\alpha=3/2$.

A \emph{$k$-cut} of an undirected graph $G=(V,E)$ is a minimal edge set whose removal leaves at least $k$ connected components.
For $|V|\ge k$, write $\lambda_k(G)$ for the minimum cost of a $k$-cut.
In a unit-cost cycle, every choice of $k$ edges is a minimum $k$-cut, giving a lower bound of $\binom nk$.
For $\lambda_k>0$ and fixed $\alpha\ge1$, Karger \cite{karger1999} extended random contraction to bound the number of $k$-cuts of cost at most $\alpha\lambda_k$ by $O(n^{2\alpha(k-1)})$ and conjectured $O(n^{\alpha k})$.
Chekuri, Quanrud and Xu \cite{chekuriQuanrudXu2020} used tree packing to obtain $O(n^{\lfloor2\alpha(k-1)\rfloor})$, and Gupta, Harris, Lee and Li \cite{guptaHarrisLeeLi2022} subsequently proved Karger's $O(n^{\alpha k})$ conjecture via a refined contraction analysis.
Like the contraction bound $O(n^{2\alpha})$ for $2$-cuts, this leaves the continuous exponent $\alpha k$.
We show that approximate $k$-cuts also exhibit a stair-step behavior and establish its strict version $O(n^{\lceil\alpha k\rceil-1})$.
Call the set of edges with endpoints in different parts of a vertex partition a \emph{crossing edge set}.
We prove the strict bound for distinct crossing edge sets of arbitrary vertex partitions.

\begin{theorem}\label{thm:main}
Fix an integer $k\ge2$ and a real $\alpha\ge1$.
In any $n$-vertex graph $G$ with $n\ge k$ and nonnegative edge costs, the number of distinct crossing edge sets of cost strictly below $\alpha\lambda_k(G)$ is $O(n^{\lceil\alpha k\rceil-1})$.
\end{theorem}

For $k=2$, this resolves the conjecture of Henzinger and Williamson \cite{henzingerWilliamson1996} and the open question of Karger~\cite[Section~3.2]{karger2000} on the number of cuts of cost strictly below $\alpha\lambda_2$ for half-integers $\alpha>1$.
When $\lambda_k(G)>0$, applying \autoref{thm:main} with $\alpha'=(\lfloor\alpha k\rfloor+1)/k>\alpha$ shows that at most $O(n^{\lfloor\alpha k\rfloor})$ distinct crossing edge sets, and in particular at most $O(n^{\lfloor\alpha k\rfloor})$ $k$-cuts, have cost at most $\alpha\lambda_k(G)$.

The exponent $\lceil\alpha k\rceil-1$ in \autoref{thm:main} is optimal for $k$-cuts whenever $\alpha>1$ by generalizing the parallel-path construction of Henzinger and Williamson \cite{henzingerWilliamson1996}.
Put $q=\lceil\alpha k\rceil-1\ge k$, take vertices $t_1,\ldots,t_k$, join $t_i$ to $t_{i+1}$ by a path of $L\ge k$ unit-cost edges for $1\le i<k$, and join $t_k$ to $t_1$ by $q-k+1$ internally disjoint paths of $L$ unit-cost edges, giving $n=k+q(L-1)$ vertices.
Deleting $k$ edges from one path leaves $k$ components, so $\lambda_k(G)=k$.
Choosing one edge from each of the $q$ paths gives $L^q=\Theta(n^{\lceil\alpha k\rceil-1})$ distinct $k$-cuts of cost $q<\alpha\lambda_k(G)$.

\paragraph{Organization}
In \autoref{sec:preliminaries}, we reduce \autoref{thm:main} to counting bounded-palette colorings, outline the proof via a coloring recurrence, and establish our tree-packing and splitting-off tools.
In \autoref{sec:mincut-two}, we prove the bound in the easy case where every $2$-cut costs at least two, which immediately settles $k=2$.
In \autoref{sec:stopping}, we reduce general graphs for $k\ge3$ to stopping instances via a recursion that contracts or marks vertex sets with small boundary.
Finally, in \autoref{sec:preparation}, we use random contraction to eliminate small $2$-cuts outside a bounded vertex set, reducing stopping instances to a fixed-color variant of the easy case.

\section{Preliminaries and proof overview}\label{sec:preliminaries}

Throughout, graphs $G=(V,E)$ are undirected and have nonnegative edge costs.
We merge parallel edges by summing their costs and discard loops to keep graphs simple.
Write $c(F)=\sum_{e\in F}c(e)$, $\delta(X)=\{uv\in E: u\in X, v\notin X\}$ for the boundary of $X$, $d(X)=c(\delta(X))$, and $w(X,Y)$ for the cost of edges between disjoint sets $X,Y$.
We omit braces for singletons and add a subscript when the graph is not clear from context.
Write $E[X]$ for the edges with both endpoints in $X$, $G[X]$ for the induced subgraph, and $G-X=G[V\setminus X]$.

A coloring $P:V\to[K]=\{1,\ldots,K\}$ represents the partition into its nonempty color classes.
Each partition into at most $K$ parts has at most $K!$ representations, so a coloring bound with a fixed palette $K$ also bounds partitions and distinct crossing edge sets up to the constant factor $K!$.
For a partition or coloring $P$ and a vertex set $X\subseteq V$, let $c(P)$ be the total cost of crossing edges, and let $c_X(P)$ be the total cost of crossing edges in $E[X]$.

We assume $\lambda_k>0$ and scale edge costs by $k/\lambda_k$, so $\lambda_k=k$, and put $s=\lceil\alpha k\rceil$.
Then \autoref{thm:main} reduces to bounding the number of colorings with $c(P)<s$ by $O(n^{s-1})$.
We call these colorings the \emph{targets}.
The following lemma lets us fix the palette size to $K=2s$.

\begin{lemma}\label{lem:part-bound}
If $\lambda_k(G)\ge k$ and $s\ge k$, every partition of cost below $s$ has fewer than $2s$ parts.
\end{lemma}

\begin{proof}
For a partition $P$ into $u\ge k$ parts with cost $c(P)<s$, isolating the $k-1$ parts $X$ with smallest $d(X)$ gives $k\le\lambda_k(G)\le2(k-1)c(P)/u<2(k-1)s/u$, so $u<2s$.
\end{proof}

\subsection{Proof overview}\label{our-techniques}

Let $\mathcal H$ be a family of $[K]$-colorings of $V$ and let $x\in V$.
For $i\in[K]$, let $\mathcal H_i$ consist of the colorings $Q:V\setminus\{x\}\to[K]$ whose extension giving $x$ color $i$ belongs to $\mathcal H$.
Since every coloring in $m\ge1$ of the sets $\mathcal H_i$ contributes $m$ to the sum below and $1+\binom m2\ge m$ to the right-hand side, we have
\begin{equation}\label{eq:colorclass}
|\mathcal H|=\sum_{i=1}^K|\mathcal H_i|
\le\left|\bigcup_{i=1}^K\mathcal H_i\right|
+\sum_{1\le i<j\le K}|\mathcal H_i\cap\mathcal H_j|.
\end{equation}
Here $\bigcup_{i=1}^K\mathcal H_i$ counts restrictions with at least one valid color at $x$, while $\mathcal H_i\cap\mathcal H_j$ counts those for which colors $i$ and $j$ both work at $x$.
For $K=2$, \autoref{eq:colorclass} recovers the recurrence of \cite{nagamochiNishimuraIbaraki1997,henzingerWilliamson1996}.

In \autoref{sec:mincut-two}, we apply \autoref{eq:colorclass} when every $2$-cut costs exactly two, which holds automatically for $k=2$ under $\lambda_2=2$.
The key observation is that $\bigcup_{i=1}^K\mathcal H_i$ is bounded by the same counting problem on $n-1$ vertices after splitting off $x$ (\autoref{thm:refined-splitting}), whereas each $\mathcal H_i\cap\mathcal H_j$ is counted on the same graph by keeping $x$ as a \emph{terminal} restricted to the allowed colors $\{i,j\}$.
Repeating \autoref{eq:colorclass} leads to counting colorings outside a sequence $S$ of $p$ terminals, chosen in minimum-adjacency order, whose every allowed extension to $S$ costs less than $s$.
We establish a uniform $O(n^{s-1-p})$ bound on such colorings, reducing the counting to the base case $p=s-1$, where the terminals consume nearly the entire cost budget and tree packing (\autoref{thm:tree-packing}) yields an $O(1)$ bound.

For $k\ge3$, this method does not apply directly because $\lambda_k=k$ allows $2$-cuts of cost below two, causing both splitting off and the base-case analysis to fail.
In \autoref{sec:stopping}, an outer recursion repeatedly selects an inclusionwise minimal proper set $X$ of more than $k$ vertices with $d(X)<2$: targets with $c_X(P)<1$ survive contracting $X$ into $k-2$ vertices, while those with $c_X(P)\ge1$ keep $X$ as a \emph{marked region}.
With disjoint marked regions $A_1,\ldots,A_h$, we prove a bound of $O(n^{s-h-1}\prod_{i=1}^h|A_i|)$ at the base cases of this recursion in \autoref{sec:preparation}, and summing over the recursion gives the desired $O(n^{s-1})$ bound.

\subsection{Technical tools}

For $|V|\ge2$, the \emph{strength} of $G$ is $\sigma(G)=\min_P c(P)/(|P|-1)$, where $P$ ranges over partitions of $V$ into at least two nonempty parts and $|P|$ denotes the number of parts.
We use the following fractional spanning-tree packing theorem (see \cite{chekuriQuanrudXu2020}).

\begin{theorem}\label{thm:tree-packing}
There are nonnegative weights $y_T$ on the spanning trees of $G$ with $\sum_T y_T=\sigma(G)$ and $\sum_{T\ni e}y_T\le c(e)$ for every edge $e$.
If $\sigma(G)>0$, sampling $T$ with probability $y_T/\sigma(G)$ gives $\Pr(e\in T)\le c(e)/\sigma(G)$.
\end{theorem}

Thus, if every partition into $r\ge2$ nonempty parts has crossing cost at least $r-1$, then $\sigma(G)\ge1$ and the distribution satisfies $\Pr(e\in T)\le c(e)$.
In particular, this holds when every $2$-cut costs at least two.

\emph{Splitting off} a pair $ux,xv\in E$ replaces it by an edge $uv$, discarding the new edge if it is a loop.
For weighted edges, we split equal-cost pieces of the two edges.
A \emph{complete splitting off} at $x$ pairs all incident edge pieces and deletes $x$.
After a complete splitting off at $x$, the restriction of any coloring to $V\setminus\{x\}$ has no greater cost, since each crossed new edge replaces at least one crossed old edge of the same cost.

We need the following form of Eulerian splitting off, protecting an induced subgraph.
For a general framework, see Kir\'aly \cite{kiraly2007}.
We give a direct proof below.

\begin{theorem}[Refined splitting off]\label{thm:refined-splitting}
Suppose every cut of $G=(V,E)$ has cost at least two.
If $x\in V$ and $Z\subseteq V\setminus\{x\}$ satisfy $w(x,Z)\le d(x)/2$, then $x$ has a complete splitting off after which every cut still costs at least two and $G[Z]$ is unchanged.
\end{theorem}

\begin{proof}
Multiply edge costs by an integer $M$ to obtain even integers and replace each edge by parallel unit edges, so every cut has size at least $L=2M$.
In this multigraph, the theorem reduces to finding a complete splitting off at $x$ without adding edges to $G[Z]$ while maintaining $d(X)\ge L$ for all $X\subsetneq V\setminus\{x\}$.
Call such a set $X$ tight if $d(X)=L$.
Since all degrees stay even, every cut size remains even, and splitting off $xu,xv$ decreases $d(X)$ by $2$ when $u,v\in X$ and leaves $d(X)$ unchanged otherwise.
Thus $xu,xv$ can be split off unless some tight set contains both $u$ and $v$.

For any edge $xu$, the tight sets containing $u$ form a chain. For two such sets $A,B\ni u$, if both $A\setminus B$ and $B\setminus A$ are nonempty, then $2L\le d(A\setminus B)+d(B\setminus A)\le d(A)+d(B)-2=2L-2$.
If a tight set contains $u$, the largest such tight set $T$ satisfies $L\le d(T\cup\{x\})=L+d(x)-2w(x,T)$, so $w(x,T)\le d(x)/2$ and each of the $d(x)-w(x,T)\ge d(x)/2$ edges $xv$ with $v\notin T$ can be split off with $xu$.
If no tight set contains $u$, all $d(x)-1\ge d(x)/2$ other edges at $x$ can be split off with $xu$.

While $w(x,Z)>0$, pick an edge $xu$ with $u\in Z$.
Since at most $w(x,Z)-1<d(x)/2$ other edges at $x$ go to $Z$, some $xv$ with $v\notin Z$ can be split off with $xu$, which adds no edge to $G[Z]$ and reduces $w(x,Z)$ by $1$ and $d(x)$ by $2$, preserving $w(x,Z)\le d(x)/2$.
Once $w(x,Z)=0$, repeatedly split off remaining pairs at $x$, merge parallel edges, and divide edge costs by $M$.
\end{proof}

\section{\texorpdfstring{The easy case: $2$-cuts cost at least two}{The easy case: 2-cuts cost at least two}}\label{sec:mincut-two}

We first count colorings under the assumption that every $2$-cut costs at least two.
Under the normalization $\lambda_k=k$, this assumption is automatic for $k=2$.
For general $k$, $\lambda_2\ge2$ forces $\lambda_2=2$ since every minimum $k$-cut partitions $V$ into $k$ parts bounded by minimum $2$-cuts, and every crossing edge is counted twice in these $2$-cuts.
Reinelt and Wenger~\cite{reineltWenger2010} studied enumerating such partitions.
By \autoref{lem:part-bound}, the following bound with $K=2s$ gives \autoref{thm:main} whenever the assumption holds.

\begin{theorem}\label{thm:mincut-two}
Fix integers $s\ge2$ and $K\ge2$.
If every cut of an $n$-vertex graph has cost at least two, then the number of colorings $P:V\to[K]$ with $c(P)<s$ is $O(n^{s-1})$.
\end{theorem}

Let $n=|V|$ and $S=(s_1,\ldots,s_p)$ be a sequence of distinct vertices, with $0\le p\le\min\{s-1,n\}$, and also write $S$ for its vertex set, so $|S|=p$.
Put $S_{j-1}=\{s_1,\ldots,s_{j-1}\}$.
The sequence $S$ is a \emph{minimum-adjacency sequence} if each $s_j$ minimizes $w(v,S_{j-1})$ over $v\in V\setminus S_{j-1}$.
For each terminal $s_j$, fix two distinct allowed colors, and let $\mathcal F(G,S)$ consist of colorings of $V\setminus S$ whose every extension to $S$ using these allowed colors has cost below $s$.
We omit the allowed color pairs from the notation because our bound holds uniformly over all choices.
When $S$ is empty, $\mathcal F(G,S)$ is the family in \autoref{thm:mincut-two}.
We will prove the stronger bound $|\mathcal F(G,S)|=O(n^{s-1-p})$.

Whenever $\mathcal F(G,S)\ne\varnothing$, fix any $Q\in\mathcal F(G,S)$ and let $Q'$ be the random extension obtained by choosing each terminal's color independently and uniformly from its allowed pair.
Every edge in $\delta(S)\cup E[S]$ has endpoints of different colors in $Q'$ with probability at least $1/2$, while every allowed extension costs less than $s$, so $c_{V\setminus S}(Q)+(d(S)+c(E[S]))/2\le\mathbb E[c(Q')]<s$.
In particular, $d(S)+c(E[S])<2s$.

The base case $p=s-1$ is the main difficulty.
Using tree packing, we find a spanning tree $T$ with $O(s)$ edges meeting $S$ such that a constant fraction of colorings in $\mathcal F(G,S)$ cross at most one edge of the forest $T-S$.
We will show that $T-S$ has $O(s)$ components.
After rooting each component, there are only $O(1)$ ways to fix the root colors and the color below the crossed edge.
Each coloring crossing one edge $e$ of $T-S$ is then determined by $e$ and recolors the subtree $A_e$ below $e$.
Counting $\mathcal F(G,S)$ thus reduces to bounding the edges $e$ whose induced coloring lies in $\mathcal F(G,S)$ via chains and antichains of their subtrees $A_e$ under inclusion.

\begin{lemma}\label{lem:tree-reduction}
Suppose $p=s-1$ and $\mathcal F(G,S)\ne\varnothing$.
For sufficiently large $n$, we have $c(E[S])<1$ and every $Q\in\mathcal F(G,S)$ satisfies $c_{V\setminus S}(Q)<1+c(E[S])/2$.
Moreover, there is a spanning tree $T$ with at most $16s$ edges meeting $S$ such that at least $|\mathcal F(G,S)|/8$ colorings in $\mathcal F(G,S)$ cross at most one edge of $T-S$.
\end{lemma}

\begin{proof}
Every vertex has degree at least two, so $d(S)+2c(E[S])=\sum_{j=1}^p d(s_j)\ge2p$ and $(d(S)+c(E[S]))/2\ge p-c(E[S])/2$.
With $c_{V\setminus S}(Q)+(d(S)+c(E[S]))/2<s$ and $p=s-1$, this gives $c_{V\setminus S}(Q)<s-p+c(E[S])/2=1+c(E[S])/2$.
By minimum adjacency, $c(E[S])=\sum_{j=1}^p w(s_j,S_{j-1})\le p\,w(v,S)$ for every $v\in V\setminus S$.
Summing over $v\in V\setminus S$ gives $(n-p)c(E[S])\le p\,d(S)<2sp$, so $c(E[S])<1$ for sufficiently large $n$.

Since every $2$-cut costs at least two, \autoref{thm:tree-packing} gives a random spanning tree $T$ with $\Pr(e\in T)\le c(e)$ for every edge $e$.
Let $Y$ be the number of edges of $T$ meeting $S$, and for $Q\in\mathcal F(G,S)$ let $X_Q$ count the edges of $T-S$ crossed by $Q$.
Then $\mathbb E[Y]\le d(S)+c(E[S])<2s$ and $\mathbb E[X_Q]\le c_{V\setminus S}(Q)<3/2$.
By Markov's inequality, $\Pr(Y>16s)<1/8$ and $\Pr(X_Q\ge2)<3/4$, so $\Pr(Y\le16s,\ X_Q\le1)>1/8$ for each $Q\in\mathcal F(G,S)$.
Summing over $Q\in\mathcal F(G,S)$ shows that some spanning tree $T$ satisfies the lemma.
\end{proof}

Assume the hypotheses of \autoref{lem:tree-reduction} with $n$ sufficiently large, and fix a spanning tree $T$ given by the lemma.
Deleting $S$ removes at most $16s$ edges and $p=s-1$ vertices from $T$, so $T-S$ has $t\le15s+2$ components.
Root each component of $T-S$, fix the root colors of all components, and fix one component of root color $r$ together with a color $d\ne r$.
For each edge $e$ in this component, let $A_e$ be the vertex set of the subtree below $e$, let $Q_e$ be the coloring of $V\setminus S$ that gives $A_e$ color $d$ and every other vertex its component's root color, and call $e$ \emph{feasible} if $Q_e\in\mathcal F(G,S)$.
Order the edges of the component by inclusion of their subtrees $A_e$, so a \emph{chain} has nested subtrees $A_e$ and an \emph{antichain} has pairwise disjoint subtrees $A_e$.

\begin{lemma}\label{lem:feasible-chains}
Every chain of feasible edges has at most $2s$ members.
\end{lemma}

\begin{proof}
Let $e_1,\ldots,e_a$ be a chain of feasible edges, indexed so that $A_1\subsetneq\cdots\subsetneq A_a$, where $A_i=A_{e_i}$.
For $i<j$, we claim that
\[
2\le c_{V\setminus S}(Q_{e_i})+c_{V\setminus S}(Q_{e_j})
  +w(S,A_j)-w(S,A_i)<2s-d(S).
\]
For the lower bound, $Q_{e_i}$ and $Q_{e_j}$ together cross every edge of the boundary of $A_j\setminus A_i$ in $G-S$; adding the edges between $S$ and $A_j\setminus A_i$ covers $\delta(A_j\setminus A_i)$, which has cost at least $2$.
For the upper bound, $Q_{e_i}$ and $Q_{e_j}$ color $A_j\setminus A_i$ by $r$ and $d$, respectively, and agree elsewhere on $V\setminus S$.
Since $r\ne d$, we can extend $Q_{e_i}$ and $Q_{e_j}$ to $S$ using distinct allowed colors at each terminal, avoiding $r$ in $Q_{e_i}$ and $d$ in $Q_{e_j}$.
Then every edge between $S$ and $A_j\setminus A_i$ is crossed in both extensions, while every other edge in $\delta(S)$ has the same nonterminal color in both extensions and distinct terminal colors, so it is crossed at least once.
Thus edges in $\delta(S)$ contribute at least $d(S)+w(S,A_j)-w(S,A_i)$ to the sum of the two extended coloring costs, which is below $2s$ since $e_i$ and $e_j$ are feasible.
Adding the crossing costs within $V\setminus S$ and rearranging gives the upper bound.

We may assume $a\ge4$ since $s\ge2$.
If $c(E[S])=0$, then $d(S)=\sum_{v\in S}d(v)\ge2p$, so $2s-d(S)\le2$, contradicting the strict inequality above.
Thus $c(E[S])>0$.
Adding the upper bounds for $(1,a)$ and $(2,a-1)$ and subtracting the lower bounds for $(1,2)$ and $(a-1,a)$ cancels the $c_{V\setminus S}$ terms and yields $w(S,A_{a-1})-w(S,A_2)<2p-d(S)\le2c(E[S])$.
The set $A_{a-1}\setminus A_2$ contains at least $a-3$ vertices, each having $w(v,S)\ge c(E[S])/p$ by minimum adjacency.
Hence $(a-3)c(E[S])/p<2c(E[S])$, which gives $a\le2p+2=2s$.
\end{proof}

\begin{lemma}\label{lem:feasible-antichains}
Every antichain of feasible edges has fewer than $4s$ members.
\end{lemma}

\begin{proof}
For any two distinct edges $e,f$ in an antichain, $A_e$ and $A_f$ are disjoint, and $Q_e,Q_f$ together cross all edges of $\delta(A_e)\cup\delta(A_f)$ in $G-S$, counting edges between $A_e$ and $A_f$ twice.
By \autoref{lem:tree-reduction} and $c(E[S])<1$,
\begin{align*}
4\le d(A_e)+d(A_f)
&\le c_{V\setminus S}(Q_e)+c_{V\setminus S}(Q_f)+w(S,A_e)+w(S,A_f)\\
&<3+w(S,A_e)+w(S,A_f),
\end{align*}
so $w(S,A_e)+w(S,A_f)>1$.
Averaging over an antichain of size $m\ge2$ gives $m/2<\sum_e w(S,A_e)\le d(S)<2s$, so $m<4s$.
\end{proof}

\begin{lemma}\label{lem:simple-terminal}
If $p=s-1$, then $|\mathcal F(G,S)|=O(1)$.
\end{lemma}

\begin{proof}
If $\mathcal F(G,S)=\varnothing$ or $n$ is bounded, the claim is immediate.
Otherwise let $T$ be the spanning tree from \autoref{lem:tree-reduction}, so $T-S$ has $t\le15s+2$ components.
For each of the $K^t$ choices of root colors, the component choice ($t$ options), and the color $d\ne r$ below the crossed edge ($K-1$ options), the minimal subtrees $A_e$ of feasible edges form an antichain of size fewer than $4s$ by \autoref{lem:feasible-antichains}, and the path from each minimal feasible subtree to the root contains at most $2s$ feasible edges by \autoref{lem:feasible-chains}.
Thus there are fewer than $8s^2$ feasible edges per choice, so at most $K^t[1+8s^2t(K-1)]=O(1)$ colorings in $\mathcal F(G,S)$ cross at most one edge of $T-S$.
Applying \autoref{lem:tree-reduction} gives $|\mathcal F(G,S)|=O(1)$.
\end{proof}

With the base case $p=s-1$ established, we now apply the recurrence \autoref{eq:colorclass} to prove \autoref{thm:mincut-two}.

\begin{proof}[{of \autoref{thm:mincut-two}}]
Let $f(n,p)$ be the maximum size of $\mathcal F(G,S)$ over $n$-vertex graphs with minimum cut at least two, minimum-adjacency sequences $S$ of length $p$, and choices of terminal color pairs.
We prove $f(n,p)=O(n^{s-1-p})$ by downward induction on $p$.
The base case $p=s-1$ is \autoref{lem:simple-terminal}, and for $n\le4s$ we have $f(n,p)\le K^{4s}=O(1)$.

Assume $p<s-1$, $n>4s$, and $\mathcal H=\mathcal F(G,S)\ne\varnothing$, which implies $d(S)<2s$.
Choose $x\in V\setminus S$ minimizing $w(x,S)$, so the extended sequence $S+x$ is a minimum-adjacency sequence.
Since $w(x,S)\le d(S)/(n-p)<2s/(3s)<1\le d(x)/2$, we may apply \autoref{thm:refined-splitting} at $x$ with $Z=S$ to obtain an $(n-1)$-vertex graph $G'$ with minimum cut at least two and $G'[S]=G[S]$.
Because splitting off only adds edges outside $S$, each vertex $v\in V(G')\setminus S$ has nondecreasing adjacency $w(v,S_{j-1})$ to every prefix of $S$, while $w(s_j,S_{j-1})$ is unchanged, so $S$ remains a minimum-adjacency sequence in $G'$.

Now apply \autoref{eq:colorclass} to $\mathcal H=\mathcal F(G,S)$ at $x$, where $\mathcal H_i$ consists of the colorings $Q:V\setminus(S\cup\{x\})\to[K]$ whose extension by $x=i$ belongs to $\mathcal F(G,S)$.
Since splitting off at $x$ never increases the cost of any coloring on $V\setminus\{x\}$, every $Q\in\bigcup_{i=1}^K\mathcal H_i$ belongs to $\mathcal F(G',S)$, so $\left|\bigcup_{i=1}^K\mathcal H_i\right|\le f(n-1,p)$.
For each pair $1\le i<j\le K$, the intersection $\mathcal H_i\cap\mathcal H_j$ is $\mathcal F(G,S+x)$ with allowed colors $\{i,j\}$ at $x$, of size at most $f(n,p+1)$.
Thus $f(n,p)\le f(n-1,p)+\binom K2 f(n,p+1)$, which sums over $n$ to give $f(n,p)=O(n^{s-1-p})$.
Setting $p=0$ completes the proof.
\end{proof}

The approximate $2$-cut bound is an immediate consequence of \autoref{thm:mincut-two}.

\begin{corollary}\label{cor:two-cut}
Fix a real $\alpha\ge1$.
In an $n$-vertex graph, the number of $2$-cuts of cost strictly below $\alpha\lambda_2(G)$ is $O(n^{\lceil2\alpha\rceil-1})$.
\end{corollary}


\section{From the easy case to general graphs}\label{sec:stopping}
\label{sec:main-proof}

For general $k\ge3$ with $\lambda_k\ge k$, there may be $2$-cuts of cost below two, which prevents repeating the proof of \autoref{thm:mincut-two} directly since splitting off can decrease $\lambda_k$ and the base case $p=s-1$ relies on $d(X)\ge2$ for every nonempty proper subset $X$.
Fix integers $s\ge k\ge3$ and $K=2s$, and consider a proper vertex set $X$ of size $|X|>k$ with $d(X)<2$.
We partition the target $[K]$-colorings of cost below $s$ according to their internal crossing cost $c_X(P)$.
\begin{itemize}
\item If $c_X(P)<1$, we show (in \autoref{lem:cheap-set}) that $X$ can be contracted into at most $k-2$ vertices without crossing any such target, reducing the number of vertices by at least $|X|-k+2$.
\item If $c_X(P)\ge1$, we mark $X$ as a special region and record the constraint $c_X(P)\ge1$.
\end{itemize}
We always choose $X$ inclusionwise minimal among proper sets of size above $k$ with $d(X)<2$ that are disjoint from existing marked regions.
Thus, at any stage of the recursion, the marked regions $A_1,\ldots,A_h$ are pairwise disjoint and contribute at least $h$ to $c(P)$.
We will bound the targets at an instance with $n$ vertices and marked regions $A_1,\ldots,A_h$ by $O\left(n^{s-h-1}\prod_{i=1}^h|A_i|\right)$.
We first justify the contraction step for the branch $c_X(P)<1$.

\begin{lemma}\label{lem:cheap-set}
Suppose $k\ge3$, $\lambda_k(G)\ge k$, and $X$ is a nonempty proper vertex set of $G$ with $d(X)<2$.
There is a partition of $X$ into at most $k-2$ connected parts such that every coloring $P$ with $c_X(P)<1$ is constant on each part.
\end{lemma}

\begin{proof}
Starting with the trivial partition $\{X\}$, repeatedly split a part along a cut of its induced subgraph of cost below one.
If this reached $k-1$ parts, the edges of $E[X]$ between parts would have total cost below $k-2$, and removing them together with $\delta(X)$ would leave at least $k$ components at cost below $k-2+d(X)<k$, contradicting $\lambda_k(G)\ge k$.
Thus the process stops with at most $k-2$ parts, none of whose induced subgraphs has a cut of cost below one.
Each part is therefore connected, and every coloring $P$ with $c_X(P)<1$ is constant on each part.
\end{proof}

A \emph{marked region} in a graph $G$ is an inclusionwise minimal proper vertex set $A$ with $|A|>k$ and $d(A)<2$.
A \emph{stopping instance} is a graph $G$ with disjoint marked regions $A_1,\ldots,A_h$ such that no proper vertex set $X\subseteq V(G)\setminus\bigcup_{i=1}^h A_i$ satisfies $|X|>k$ and $d(X)<2$.
The following lemma bounds the base cases of the recursion.

\begin{lemma}\label{lem:stopping}
Fix integers $s\ge k\ge3$ and put $K=2s$.
Let $G$ be an $n$-vertex stopping instance with $n\ge k$, $\lambda_k(G)\ge k$, and marked regions $A_1,\ldots,A_h$.
Then the number of $[K]$-colorings $P$ with $c(P)<s$ and $c_{A_i}(P)\ge1$ for every $i\in[h]$ is $O(n^{s-h-1}\prod_{i=1}^h|A_i|)$.
\end{lemma}

A nonempty target family satisfies $h\le s-1$, since the disjoint marked regions contribute at least $h$ to $c(P)<s$.
If $n\le2k$, there are $O(1)$ colorings.
If $h=0$ and $n>2k$, every cut of $G$ has one side of size above $k$, so the stopping condition implies that every $2$-cut costs at least two and \autoref{thm:mincut-two} applies.
We prove \autoref{lem:stopping} for $n>2k$ and $1\le h<s$ in \autoref{sec:preparation}.
Assuming \autoref{lem:stopping}, we can immediately complete the proof of \autoref{thm:main}.

\begin{proof}[of \autoref{thm:main}]
Normalize edge costs so that $\lambda_k(G)=k$ and put $s=\lceil\alpha k\rceil$ and $K=2s$, as in \autoref{sec:preliminaries}.
By \autoref{lem:part-bound}, it suffices to bound the number of $[K]$-colorings of cost below $s$.
The case $k=2$ is \autoref{thm:mincut-two}, so assume $k\ge3$.

Starting from $G$ with no marked regions, recursively count colorings $P$ of the current graph $G$ with $c(P)<s$ and $c_{A_i}(P)\ge1$ for all marked regions $A_1,\ldots,A_h$, discarding empty families.
Stop when $n=|V(G)|\le2k$ or when $(G;A_1,\ldots,A_h)$ is a stopping instance.
Otherwise, choose an inclusionwise minimal set $X\subseteq V(G)\setminus\bigcup_{i=1}^h A_i$ with $|X|>k$ and $d(X)<2$, and branch on $c_X(P)$:
\begin{itemize}
\item If $c_X(P)<1$, refine the partition of $X$ from \autoref{lem:cheap-set} into $k-2$ nonempty parts and contract each part.
Since every target in this branch is constant on each part, it induces a valid target in the resulting quotient graph on $n-\Delta$ vertices, where $\Delta=|X|-k+2\ge3$, with the same marked regions $A_1,\ldots,A_h$.
\item If $c_X(P)\ge1$, add $A_{h+1}=X$ to the marked regions on $G$.
\end{itemize}
Contraction preserves $\lambda_k\ge k$ whenever $n-\Delta\ge k$, and preserves the minimality of existing marked regions $A_i$ since it does not alter $A_i$ or $d(Y)$ for any $Y\subseteq A_i$.

To sum the bounds over the recursion tree, define the potential
\[\Phi(n;A_1,\ldots,A_h)=n^{s-h-1}\prod_{i=1}^h(|A_i|-k+2).\]
Since $|A_i|>k$, we have $|A_i|<(k-1)(|A_i|-k+2)$, so \autoref{lem:stopping} bounds every stopping leaf by $O((k-1)^h\Phi)=O(\Phi)$, and leaves with $n\le2k$ also have $O(1)=O(\Phi)$ targets.
At an internal node with $d=s-h-1\ge1$, the two children have potentials $(n-\Delta)^d\prod_{i=1}^h(|A_i|-k+2)$ and $\Delta n^{d-1}\prod_{i=1}^h(|A_i|-k+2)$, which sum to at most $\Phi$ since $(n-\Delta)^d+\Delta n^{d-1}\le n^d$.
If $d=0$, the marking branch is empty (as $h+1=s$ would force $c(P)\ge s$) and the contraction branch has potential $\Phi$.
Thus the sum of $\Phi$ over all leaves is at most its initial value $n^{s-1}$, giving $O(n^{s-1})$ targets in total.
\end{proof}

\section{Counting colorings at a stopping instance}\label{sec:preparation}

In this section we prove \autoref{lem:stopping}. The general idea is to remove small $2$-cuts in the graph and apply a varient of \autoref{thm:mincut-two}.
In \autoref{sec:freezing}, we run a random contraction process within each region to eliminate small $2$-cuts outside a vertex set $B$ with bounded size, reducing to instances where $d(X)\ge2$ for every nonempty $X\subseteq V\setminus B$.
In \autoref{sec:regional}, we extend \autoref{thm:mincut-two} to graphs with fixed colors on $B$ (\autoref{cor:prescribed}) and then use spanning forests to obtain a factor of $|A_i|/n$ from each remaining marked region $A_i$ (\autoref{lem:localized}).

\subsection{Preparation by freezing and contraction}\label{sec:freezing}

Let $\mathcal H$ be the target family in \autoref{lem:stopping}, and assume $\mathcal H\ne\varnothing$, $n>2k$, and $1\le h<s$.
Write $A_0=V(G)\setminus\bigcup_{i=1}^h A_i$, and put $\varepsilon=\frac{1}{4(k-1)}$, $p_0=\frac{1}{2n}$, and $p_i=\frac{1}{2|A_i|}$ for $i\in[h]$.

Call a vertex of a contracted graph \emph{small} if its preimage in $G$ has at most $k$ vertices, and \emph{large} otherwise.
By the stopping condition and the minimality of $A_1,\ldots,A_h$, every proper subset of $A_i$ ($i\ge1$) or subset of $A_0$ of size above $k$ has boundary cost at least two in $G$.
To eliminate small vertices, we contract edges within each $A_i$ ($0\le i\le h$) until only $O(1)$ small vertices remain.
A target $P$ survives a contraction at $v$ only if the contracted edge does not cross $P$, which is unlikely when almost all edge cost from $v$ to its neighbors in $A_i$ crosses $P$.
So besides contracting, we may \emph{freeze} a vertex so that it is never selected as an endpoint of a contracted edge.
Thus a vertex $v\in A_i$ can be merged only with its unfrozen neighbors in $A_i$.

Specifically, repeatedly select a small unfrozen vertex $v$ in some $A_i$ ($0\le i\le h$) whose edges to unfrozen neighbors in $A_i$ have total cost at least $1+2\varepsilon$.
With probability $p_i$, freeze $v$; with probability $1-p_i$, contract an edge from $v$ to an unfrozen neighbor in $A_i$ (which may be small or large), chosen with probability proportional to edge cost.
Stop when no such vertex remains.
Each step either freezes a vertex or merges $v$ with another unfrozen vertex, so the process terminates in at most $n$ steps.

Let $\mathcal H'$ consist of the targets $P\in\mathcal H$ for which no crossing edge is contracted and, at each step that freezes a vertex $v\in A_i$, the edges from $v$ to currently unfrozen neighbors of the same $P$-color in $A_i$ have total cost below $\varepsilon$.
Let $r_i$ be the number of vertices frozen in $A_i$, and define the weight $W=\prod_{i=0}^h p_i^{-r_i}=(2n)^{r_0}\prod_{i=1}^h(2|A_i|)^{r_i}$.
We will bound $W|\mathcal H'|$ for every outcome of the process; the following lemma then transfers this bound to $|\mathcal H|$.

\begin{lemma}\label{lem:weighted-survival}
For some constant $\rho>0$, $\mathbb E[W|\mathcal H'|]\ge\rho|\mathcal H|$.
\end{lemma}

\begin{proof}
It suffices to prove $\mathbb E[W\mathbf 1_{\{P\in\mathcal H'\}}]\ge\rho$ for each fixed $P\in\mathcal H$, with expectation taken over the contraction process.
To bound this expectation, we introduce a modified process depending on $P$ that ensures $P\in\mathcal H'$ in every run.
It uses the same rule to select a small unfrozen vertex $v\in A_i$ at each step.
Let $a$ and $b$ be the costs of edges from $v$ to unfrozen neighbors in $A_i$ of different and equal $P$-color, respectively.
If $b<\varepsilon$, freeze $v$; if $b\ge\varepsilon$, contract an edge from $v$ to a same-color unfrozen neighbor in $A_i$, chosen with probability proportional to cost.
Every run of the modified process places $P$ in $\mathcal H'$.

Fix a run of the modified process and compare its probability under the two processes.
At each freezing step in $A_i$, conditional on the preceding steps, the original process freezes the selected vertex with probability $p_i$, whereas the modified process freezes it with probability one.
Thus the freezing steps contribute a factor $\prod_{i=0}^h p_i^{r_i}=W^{-1}$ to the ratio of the original run probability to the modified run probability.
Multiplying the original run probability by $W$ cancels this factor, leaving only the probability ratios from contraction steps.
Contracting a same-color edge $e$ has probability $(1-p_i)\frac{c(e)}{a+b}$ in the original process and $\frac{c(e)}{b}$ in the modified one, with ratio $(1-p_i)\frac{b}{a+b}$.
Let $q_i$ be the number of contractions in $A_i$ along the run.
Since $q_0\le n$ and $q_i\le|A_i|$ for $i\ge1$, Bernoulli's inequality gives $(1-p_i)^{q_i}\ge1-q_ip_i\ge1/2$ for every $i$, so $\prod_{i=0}^h(1-p_i)^{q_i}\ge2^{-(h+1)}\ge2^{-s}$.
Each contraction at $v$ increases the size of its preimage in $G$ by at least one, so each endpoint of an original crossing edge of $P$ is selected as the small vertex $v$ at most $k$ times.
Hence $\sum a\le 2k\,c(P)<2ks$ over all contractions in the run.
Since $b\ge\varepsilon$ at every contraction, $\frac{b}{a+b}\ge\exp(-a/\varepsilon)$, so the factors $\frac{b}{a+b}$ multiply to at least $\exp(-2ks/\varepsilon)$.
Every modified run therefore occurs in the original process with $W$-weighted probability at least $\rho=2^{-s}\exp(-2ks/\varepsilon)$ times its modified probability.
Summing over all runs gives $\mathbb E[W\mathbf 1_{\{P\in\mathcal H'\}}]\ge\rho$, and summing over $P\in\mathcal H$ proves the lemma.
\end{proof}

\paragraph{Reducing to regional counting.}
Fix an outcome of the process with $\mathcal H'\ne\varnothing$, and fix any $P\in\mathcal H'$.
Let $G'$ be the resulting graph, $A_i'$ the image of $A_i$, $Z$ the set of all frozen vertices, and $r=|Z|=\sum_{i=0}^h r_i$.
When a vertex $v\in A_i$ is frozen, let $E_v$ be the set of edges between $v$ and currently unfrozen neighbors in $A_i$; then $c(E_v)\ge1+2\varepsilon$, and the edges of $E_v$ uncrossed by $P$ cost less than $\varepsilon$, so the crossed edges in $E_v$ cost more than $c(E_v)-\varepsilon\ge1+\varepsilon$.
Once $v$ is frozen, it is no longer an unfrozen neighbor at later steps, so the edge sets $E_v$ ($v\in Z$) are pairwise disjoint.
Thus $(1+\varepsilon)r\le\sum_{v\in Z}(c(E_v)-\varepsilon)<c(P)<s$, which gives $r<s$ and $\sum_{v\in Z}c(E_v)<s+r\varepsilon=O(1)$.
Every edge in $\delta_{G'}(Z)$ belongs to $\bigcup_{v\in Z}E_v$ or to $\delta_G(A_i)$ for some $i\ge1$, so $d_{G'}(Z)\le\sum_{v\in Z}c(E_v)+\sum_{i=1}^h d_G(A_i)=O(1)$.

Let $U$ be the set of small unfrozen vertices in $G'$.
We claim that $|U|=O(1)$.
Assume $|U|\ge k$, isolating the preimages of the $k-1$ vertices of smallest degree in $U$ gives $k\le\frac{k-1}{|U|}\sum_{v\in U}d_{G'}(v)$, so $\sum_{v\in U}d_{G'}(v)\ge(1+4\varepsilon)|U|$.
At termination, each vertex of $U$ has edge cost below $1+2\varepsilon$ to unfrozen neighbors in its own region; all remaining incident edges have total cost at most $d_{G'}(Z)+\sum_{i=1}^h d_G(A_i)=O(1)$.
Thus $\sum_{v\in U}d_{G'}(v)\le(1+2\varepsilon)|U|+O(1)$, giving $2\varepsilon|U|=O(1)$ and hence $|U|=O(1)$.

Let $B\subseteq V(G')$ consist of $Z$, $U$, and one large vertex from each $A_i'$ ($i\ge1$) that contains a large vertex.
We have $|B|\leq |Z|+|U|+h=O(1)$.
We claim that $d_{G'}(X)\ge2$ for every nonempty $X\subseteq V(G')\setminus B$.
Because $X$ avoids $B$, every vertex in $X$ is large, and any $A_i'$ ($i\ge1$) that meets $X$ also contains a large vertex in $B\setminus X$.
Consequently:
\begin{itemize}
\item If $X\subseteq A_0'$, its preimage in $G$ is a subset of $A_0$ of size above $k$, so the stopping condition gives $d_{G'}(X)\ge2$.
\item If $X\subseteq A_i'$ for some $i\ge1$, its preimage in $G$ is a proper subset of $A_i$ of size above $k$, so the minimality of $A_i$ gives $d_{G'}(X)\ge2$.
\item If $X$ meets some $A_i'$ ($i\ge1$) and $X\setminus A_i'\ne\varnothing$, then $A_i'\setminus X$ contains the large vertex in $B\cap A_i'$, so its preimage is a proper subset of $A_i$ of size above $k$.
Minimality of $A_i$ gives $d_{G'}(A_i'\setminus X)\ge2>d_{G'}(A_i')$, and posimodularity yields $d_{G'}(X\setminus A_i')\le d_{G'}(X)+d_{G'}(A_i')-d_{G'}(A_i'\setminus X)<d_{G'}(X)$.
Removing each such $A_i'$ one by one reduces $X$ to a nonempty subset of a single $A_j'$ ($0\le j\le h$) while decreasing $d_{G'}$, proving $d_{G'}(X)\ge2$.
\end{itemize}

Recall that we need an upper bound on $W|\mathcal H'|$. Each freeze increases $W$ by a factor $2n$ or $2|A_i|$ and cover at least $1+\varepsilon$ crossing cost for any $P\in \mathcal H'$.
We delete the edges incident to $Z$ and count the resulting colorings with the lower threshold $t=s-r\ge1$.
If $r=0$, set $H=G'$ and $\gamma=1$.
Otherwise, set $\gamma=1+\frac{\varepsilon}{2s}$ and form $H$ by deleting all edges incident to $Z$ and scaling the remaining costs by $\gamma$.
Every $P\in\mathcal H'$ induces a coloring of $V(G')=V(H)$; since no crossing edge was contracted, its crossing cost in $G'$ is still $c(P)<s$.
When $r>0$, deleting the edges incident to $Z$ removes more than $(1+\varepsilon)r$ of this crossing cost, leaving less than $s-(1+\varepsilon)r$ on the remaining edges.
Scaling these remaining costs by $\gamma$ gives the crossing cost of $P$ in $H$, which satisfies $c_H(P)<\gamma(s-(1+\varepsilon)r)\le\gamma(t-\varepsilon)<t$.

Deleting edges may create nonempty sets $X\subseteq V(H)\setminus B$ with $d_H(X)<2$.
We construct a hitting set of size $O(1)$ for these sets and add it to $B$, so that all nonempty sets outside the enlarged $B$ have boundary cost at least two.
If $r=0$, the bound already holds, so assume $r>0$.
Before deleting the edges incident to $Z$, every such set $X$ satisfies $d_{G'}(X)\ge2$.
Hence $w_{G'}(X,Z)=d_{G'}(X)-d_H(X)/\gamma>2-2/\gamma=\Omega(1)$.
Thus scaling ensures that every violating set lost at least a fixed positive amount of boundary cost to $Z$.
The inclusionwise minimal violating sets are pairwise disjoint: if two distinct such sets $X,Y$ overlapped, minimality and posimodularity would give $4\le d_H(X\setminus Y)+d_H(Y\setminus X)\le d_H(X)+d_H(Y)<4$.
Their edges to $Z$ are therefore disjoint and have total cost at most $d_{G'}(Z)=O(1)$, so there are only $O(1)$ minimal violating sets.
Choosing one vertex from each minimal violating set gives a hitting set.
Adding this hitting set to $B$ preserves the cut lower bound.

Deleting edges incident to $Z$ may remove internal crossing edges from a marked region, so its internal crossing cost need no longer be at least one.
We retain this constraint only for regions containing no frozen vertex, indexed by $J=\{i\in[h]:r_i=0\}$.
For each $j\in J$, no internal edge of $A_j'$ was incident to $Z$, so $c_{A_j'}(P)\ge1$ in $H$ for every $P\in\mathcal H'$ and $d_H(A_j')<2\gamma=O(1)$.
Enumerating the $K^{|B|}=O(1)$ colorings of $B$ leaves $O(1)$ counting problems on $H$ with fixed colors on $B$, threshold $t=s-r$, and regional constraints for $\{A_j':j\in J\}$.

\subsection{Counting with regional constraints}\label{sec:regional}

We first extend \autoref{thm:mincut-two} to graphs where the cut lower bound $d(X)\ge2$ holds outside a bounded set $B$ of vertices with fixed colors.

\begin{corollary}\label{cor:prescribed}
Fix nonnegative integers $s,K,D$.
Let $G=(V,E)$ be an $n$-vertex graph with fixed colors on $B\subseteq V$, where $|B|\le D$, such that $d(X)\ge2$ for every nonempty $X\subseteq V\setminus B$.
The number of $[K]$-colorings of $V$ extending $B$ with cost below $s$ is $O(n^{s-1})$.
\end{corollary}

\begin{proof}
Assume the family of such colorings is nonempty; since $V\ne\varnothing$ and $d(V)=0$, we have $B\ne\varnothing$ and $s,K\ge1$.
Form $G'$ by adding a new vertex $z$ and joining each $b\in B$ to $z$ by two internally disjoint paths of length two with unit-cost edges, whose internal vertices form an independent set $S$ of size $2|B|$.
Any cut of $G'$ separating some $b\in B$ from $z$ crosses both paths between $b$ and $z$, while any other cut has a side disjoint from $B\cup\{z\}$, which either meets $V\setminus B$ (cost at least two in $G$) or is a nonempty subset of $S$ (each vertex of $S$ has degree two).
Thus every cut of $G'$ costs at least two.
Because $S$ is an independent set, $w(s_j,S_{j-1})=0$ for any ordering of $S$, so $S$ is a minimum-adjacency sequence in $G'$.
Fix color $0\notin[K]$ at $z$, and for each internal vertex $u\in S$ on a path between $b\in B$ and $z$, allow $u$ either color $0$ or the fixed color of $b$.
Then for any coloring $P$ of $V$ extending $B$, setting $z$ to color $0$ ensures that either allowed color at $u\in S$ matches one neighbor on its path and differs from the other, crossing exactly one of the two unit edges.
Hence every allowed assignment on $S$ has cost $c(P)+|S|<s+|S|$ in $G'$.
Applying the bound $|\mathcal F(G',S)|=O(|V(G')|^{(s+|S|)-1-|S|})=O(n^{s-1})$ from the proof of \autoref{thm:mincut-two} with palette $[K]\cup\{0\}$ and threshold $s+|S|$ proves the corollary.
\end{proof}

We now refine \autoref{cor:prescribed} to save a factor of $|R_i|/n$ for each disjoint region $R_i$ with required internal crossing cost $c_{R_i}(P)\ge1$.

\begin{lemma}\label{lem:localized}
Fix nonnegative integers $s,K,h,D$.
Let $G=(V,E)$ be an $n$-vertex graph with fixed colors on a nonempty set $B\subseteq V$ of size at most $D$, such that $d(X)\ge2$ for every nonempty $X\subseteq V\setminus B$.
If $R_1,\ldots,R_h\subseteq V$ are disjoint nonempty sets with $d(R_i)\le D$, then the number of $[K]$-colorings extending $B$ with $c(P)<s$ and $c_{R_i}(P)\ge1$ for every $i\in[h]$ is $O(n^{s-h-1}\prod_{i=1}^h|R_i|)$.
\end{lemma}

\begin{proof}
We induct on $h$.
The case $h=0$ is \autoref{cor:prescribed}.
Assume $h\ge1$ and that the target family is nonempty, so $h\le\sum_{i=1}^h c_{R_i}(P)\le c(P)<s$ and $D\ge|B|\ge1$.
We construct a random spanning forest $T$ on $V$ with exactly one vertex of $B$ per component and $\Pr(e\in T)\le c(e)$ for every $e\in E$.
Contract $B$ to a vertex $z$, discard loops and merge parallel edges. In particular, the edge $zv$ has cost $w(B,v)$ and the cost of every cut of this graph is at least two.
By \autoref{thm:tree-packing}, there is a random spanning tree $T'$ in which each edge is selected with probability at most its cost.
Replace each edge $zv\in T'$ by an original edge $bv$, choosing $b\in B$ with probability $c(bv)/w(B,v)$.
Each component of $T'-z$ has exactly one edge to $z$, so the replacement attaches it to exactly one vertex of $B$.
The chosen edges on $V$ therefore form the required spanning forest $T$, with unused vertices of $B$ isolated.
For $w(B,v)>0$, we have $\Pr(bv\in T)=\Pr(zv\in T')c(bv)/w(B,v)\le c(bv)$.

Let $Y=\bigl|T\cap\bigcup_{i=1}^h\delta(R_i)\bigr|$, so $\mathbb E[Y]\le\sum_{i=1}^h d(R_i)\le hD$.
For each target $P$, let $X(P)$ count the edges of $T$ crossed by $P$, so $\mathbb E[X(P)]\le c(P)<s$.
Set $L=2hD(s+1)=O(1)$.
By Markov's inequality, $\Pr(X(P)\ge s+1)<s/(s+1)$ and $\Pr(Y>L)\le hD/L=1/(2(s+1))$, so $\Pr(X(P)\le s,\,Y\le L)>1/(2(s+1))$.
Summing over targets gives a forest $T$ with $Y\le L$ that satisfies $X(P)\le s$ for a constant fraction of targets. Fix such a forest $T$; it suffices to bound the targets with $X(P)\le s$.

Put $R_0=V\setminus\bigcup_{i=1}^h R_i$ and $F_i=T[R_i]$ for $0\le i\le h$.
Deleting the $Y$ edges of $T\cap\bigcup_{i=1}^h\delta(R_i)$ from $T$ leaves the forests $F_0,\ldots,F_h$, which therefore have $|B|+Y\le D+L=O(1)$ components in total.
For each target $P$ with $X(P)\le s$, let $X_i$ count its crossed edges in $F_i$, so $\sum_{i=0}^h X_i\le X(P)\le s$.
Every such target either satisfies $X_i\le1$ for some $i\in[h]$ or satisfies $X_i\ge2$ for all $i\in[h]$:
\begin{itemize}
\item Fix $i\in[h]$ and consider the targets with $X(P)\le s$ and $X_i\le1$.
Deleting the at most one crossed edge of $F_i$ leaves $O(1)$ monochromatic components on $R_i$, so these targets induce $O(|R_i|)$ colorings on $R_i$.
Fix one such coloring $Q:R_i\to[K]$, which agrees with $B$ on $B\cap R_i$ since targets extend $B$.
Delete $E[R_i]$, contract each color class of $Q$ in $R_i$ into a vertex of that color, and replace $B\cap R_i$ in $B$ by these at most $K$ vertices.
Each target $P$ extending $Q$ has $c_{R_i}(P)\ge1$, so it induces a distinct $[K]$-coloring of the resulting at most $n$-vertex graph extending $B$ with cost below $s-1$.
Since $R_1,\ldots,R_h$ are disjoint, each $R_j$ with $j\in[h]\setminus\{i\}$ is unchanged, keeping $d(R_j)\le D$ and internal crossing cost at least one, while every nonempty set avoiding the updated $B$ also avoids $R_i$ and the original $B$, keeping boundary cost at least two.
With $1\le|B|\le D+K$ and $s-h-1\ge0$, the induction hypothesis for $h-1$ regions and threshold $s-1$ gives $O(n^{s-h-1}\prod_{j\in[h]\setminus\{i\}}|R_j|)$ such targets.
Multiplying by the $O(|R_i|)$ choices of $Q$ and summing over $i\in[h]$ gives $O(n^{s-h-1}\prod_{i=1}^h|R_i|)$.

\item Consider the targets with $X(P)\le s$ and $X_i\ge2$ for all $i\in[h]$, and permute $[h]$ so that $|R_1|=\min_{i\in[h]}|R_i|$.
For each $i\in\{0,\ldots,h\}\setminus\{1\}$, independently choose $\min\{|E(F_i)|,|R_1|+s\}$ edges of $F_i$ uniformly at random and contract all other edges of $F_i$ in $G$, leaving $R_1$ unchanged.
Since each $F_i$ is a forest on $R_i$ with $O(1)$ components, its contracted edges form $O(|R_1|)$ components on $R_i$, so the contracted graph has $O(|R_1|)$ vertices.
Because $R_0,\ldots,R_h$ are pairwise disjoint and each component of $T$ contains one vertex of $B$, no two vertices of $B$ are merged.
Contracting edges preserves cut costs, so in every outcome \autoref{cor:prescribed} bounds the number of $[K]$-colorings of the contracted graph extending $B$ with cost below $s$ by $O(|R_1|^{s-1})$.

For any such target $P$ and each $2\le i\le h$, all $X_i\le s$ crossed edges of $F_i$ are among the chosen uncontracted edges with probability $1$ if $|E(F_i)|\le|R_1|+s$, and with probability $\prod_{\ell=0}^{X_i-1}\frac{|R_1|+s-\ell}{|E(F_i)|-\ell}\ge(|R_1|/|R_i|)^{X_i}$ otherwise, since $|E(F_i)|<|R_i|$.
Because $|R_1|\le|R_i|\le n$ and $X_i\ge1$, this probability is always at least $\frac{|R_1|}{|R_i|}(|R_1|/n)^{X_i-1}$.
Similarly, since $|E(F_0)|<n$, all $X_0\le s$ crossed edges of $F_0$ are uncontracted with probability at least $(|R_1|/n)^{X_0}$.
From $X_1\ge2$ and $\sum_{i=0}^h X_i\le s$, we have $X_0+\sum_{i=2}^h(X_i-1)\le s-h-1$, so no crossed forest edge of $P$ is contracted with probability at least $(|R_1|/n)^{s-h-1}\prod_{i=2}^h\frac{|R_1|}{|R_i|}=|R_1|^{s-1}/\bigl(n^{s-h-1}\prod_{i=1}^h|R_i|\bigr)$.
Whenever this holds, $P$ induces a distinct coloring counted above, so dividing the expected bound $O(|R_1|^{s-1})$ by this probability gives $O(n^{s-h-1}\prod_{i=1}^h|R_i|)$ such targets.
\end{itemize}
Combining both cases completes the induction.
\end{proof}

Combining \autoref{lem:weighted-survival} and \autoref{lem:localized} now completes the proof of \autoref{lem:stopping}.

\begin{proof}[{of \autoref{lem:stopping}}]\label{sec:stopping-proof}
In any outcome of the random contraction process in \autoref{sec:freezing} with $\mathcal H'\ne\varnothing$, each of the $O(1)$ instances of \autoref{lem:localized} on $H$ has at most $n$ vertices, threshold $c(P)\le s-r$, and surviving regions $\{A_j':j\in J\}$ of sizes $|A_j'|\le|A_j|$.
Since any $P\in\mathcal H'$ satisfies $|J|\le\sum_{j\in J}c_{A_j'}(P)\le c_H(P)<s-r$, we have $s-r-|J|-1\ge0$, and \autoref{lem:localized} gives $|\mathcal H'|=O(n^{s-r-|J|-1}\prod_{j\in J}|A_j|)$.
For each $i\in[h]\setminus J$, we have $r_i\ge1$ and hence $|A_i|^{r_i}\le|A_i|n^{r_i-1}$.
Since $r<s$ and $|[h]\setminus J|=h-|J|$, the weight satisfies $W=2^r n^{r_0}\prod_{i\in[h]\setminus J}|A_i|^{r_i}=O(n^{r-h+|J|}\prod_{i\in[h]\setminus J}|A_i|)$.
Multiplying gives $W|\mathcal H'|=O(n^{s-h-1}\prod_{i=1}^h|A_i|)$ in every outcome, including when $\mathcal H'=\varnothing$.
Taking expectations and applying \autoref{lem:weighted-survival} yields $|\mathcal H|=O(n^{s-h-1}\prod_{i=1}^h|A_i|)$.
\end{proof}

\paragraph*{AI Disclosure}
An initial proof of the main result was discovered by Xean~\cite{Xu2026Xean} with GPT-6~Astra and required extensive technical patches.
The authors subsequently developed the present framework and rewrote the proof with assistance from AI tools.
This use of generative AI primarily affected \autoref{sec:stopping} and \autoref{sec:preparation}.
The authors take full responsibility for the correctness and content of the paper, including material generated with AI assistance.

\bibliography{references.bib}
\bibliographystyle{splncs04}

\end{document}